\documentclass{elsart}

\journal{Journal of Number Theory}

\usepackage{amssymb}
\usepackage{amsxtra}
\usepackage[mathscr]{eucal}
\usepackage{graphics}
\usepackage{epsfig}

\newtheorem{theorem}{Theorem}[section]
\newtheorem{lemma}[theorem]{Lemma}
\newtheorem{corollary}[theorem]{Corollary}
\newtheorem{proposition}[theorem]{Proposition}

\newtheorem{remark}[theorem]{Remark}
\theoremstyle{definition}

\newtheorem{definition}[theorem]{Definition}
\newtheorem{conjecture}[theorem]{Conjecture}

\newenvironment{proof}{\smallskip\noindent{\sc{Proof.}}}{\hfill$\square$\medskip}

\def\notdivides{\mathrel{\kern-3pt\not\!\kern3.5pt\bigm|}}

\begin{document}

\begin{frontmatter}

\title{On the coefficients of the cyclotomic polynomials of order three}

\author{Jia Zhao\corauthref{cor}},
\ead{zhaojia@mails.tsinghua.edu.cn} \corauth[cor]{Corresponding
author. Department of Mathematical Sciences, Tsinghua University,
Beijing 100084, China}
\author{Xianke Zhang}.

\ead{xzhang@math.tsinghua.edu.cn}

\begin{abstract}{We say that a cyclotomic polynomial $\Phi_{n}(x)$ has order three if $n$ is the product
of three distinct primes, $p<q<r$. Let $A(n)$ be the largest
absolute value of a coefficient of $\Phi_{n}(x)$ and $M(p)$ be the
maximum of $A(pqr)$. In 1968, Sister Marion Beiter [3, 4]
conjectured that $A(pqr)\leqslant \frac{p+1}{2}$. In 2008, Yves
Gallot and Pieter Moree [8] showed that the conjecture is false for
every $p\geqslant 11$, and they proposed the Corrected Beiter
conjecture: $M(p)\leqslant \frac{2}{3}p$. Here we will give a
sufficient condition for the Corrected Beiter conjecture and prove
it when $p=7$.}
\end{abstract}

\end{frontmatter}

\section{Introduction}
The $n$th cyclotomic polynomial is the monic polynomial whose roots
are the primitive $n$th roots of unity and are all simple. It is
defined by
$$\Phi_{n}(x)=\prod_{\substack{1\leqslant a\leqslant n\\(a,n)=1}}(x-e^{\frac{2\pi
ia}{n}})=\sum_{i=0}^{\phi(n)}c_{i}x^{i}.\eqno(1.1)$$ The degree of
$\Phi_{n}$ is $\phi(n)$, where $\phi$ is the Euler totient function.
It is known that the coefficients $c_{i}$, where $0\leqslant
i\leqslant \phi(n)$, are all integers.
\begin{definition}
$$A(n)=\max\{|c_{i}|,0\leqslant
i\leqslant \phi(n)\}.\eqno(1.2)$$
\end{definition}
For $n<105$, $A(n)=1$. It was once conjectured that this would hold
for all $n$, however $A(105)=2$. Note that 105 is the smallest
positive integer that is the product of three distinct odd primes.
In fact, it is easy to prove that $A(p)=1$ and $A(pq)=1$ for
distinct primes $p,q$. Besides, we have the following useful
propositions.
\begin{proposition}
The nonzero coefficients of $\Phi_{pq}(x)$ alternate between $+1$
and $-1$.
\end{proposition}
\begin{proposition} Let $p$ be a prime.\\
If $p\,|\,n$, then $\Phi_{pn}(x)=\Phi_{n}(x^{p})$, so $A(pn)=A(n)$.\\
If $p\nmid n$, then $\Phi_{pn}(x)=\Phi_{n}(x^{p})/\Phi_{n}(x)$.\\
If $n$ is odd, then $\Phi_{2n}(x)=\Phi_{n}(-x)$, so $A(2n)=A(n)$.
\end{proposition}
\begin{proof}
See [11] for details.
\end{proof}

By the proposition above, it suffices to consider squarefree values
of $n$ to determine $A(n)$. For squarefree $n$, the number of
distinct odd prime factors of $n$ is the order of the cyclotomic
polynomial $\Phi_{n}$. Therefore the cyclotomic polynomials of order
three are the first non-trivial case with respect to $A(n)$. We also
call them ternary cyclotomic polynomials.

Assume $p<q<r$ are odd primes, Bang [2] proved the bound
$A(pqr)\leqslant p-1$. This was improved by Beiter [3, 4], who
proved that $A(pqr)\leqslant p-\lfloor\frac{p}{4}\rfloor$, and made
the following conjecture.
\begin{conjecture}[Beiter]$A(pqr)\leqslant \frac{p+1}{2}$.
\end{conjecture}
Beiter proved her conjecture for $p\leqslant 5$ and also in case
either $q$ or $r\equiv\pm 1\pmod{p}$ [3]. If this conjecture holds,
it is the strongest possible result of this form. This is because
M\"{o}ller [12] indicated that for any prime $p$ there are
infinitely many pairs of primes $q<r$ such that $A(pqr)\geqslant
\frac{p+1}{2}$. Define
$$M(p)=\max\{A(pqr)\mid p<q<r\},$$ where the prime $p$ is fixed, and $q$ and
$r$ are arbitrary primes. Now with M\"{o}ller's result, we can
reformulate Beiter's conjecture.
\begin{conjecture}
For $p>2$, we have $M(p)=\frac{p+1}{2}$.
\end{conjecture}
However, Gallot and Moree [8] showed that Beiter's conjecture is
false for every $p\geqslant 11$. For $p=7$, it is still an open
problem. In this paper, we will give an answer. Based on extensive
numerical computations, they gave many counter-examples and proposed
the Corrected Beiter conjecture.
\begin{conjecture}[Corrected Beiter conjecture]
We have $M(p)\leqslant \frac{2}{3}p$.
\end{conjecture}
This is the strongest corrected version of Beiter's conjecture
because they also proved that for any $\varepsilon >0$,
$\frac{2}{3}p(1-\varepsilon)\leqslant M(p)\leqslant \frac{3}{4}p$
for every sufficiently large prime $p$.

\section{Preliminaries}
Let $p<q<r$ be odd primes. We will first give a lemma for computing
the coefficients of $\Phi_{pqr}$ explicitly. By Proposition 1.2, we
can get
$$\Phi_{pqr}(x)=\frac{\Phi_{pq}(x^{r})}{\Phi_{pq}(x)}=\frac{\Phi_{pq}(x^{r})(x^{p}-1)(x^{q}-1)}{(x^{pq}-1)(x-1)}=\sum_{i}c_{i}x^{i}.\eqno(2.1)$$
Let
$$f(x)=\Phi_{pqr}(x)(x-1)=\sum_{i}(c_{i-1}-c_{i})x^{i}=\sum_{j}b_{j}x^{j}\eqno(2.2)$$
and
$$g(x)=f(x)(x^{pq}-1)=\sum_{j}(b_{j-pq}-b_{j})x^{j}=\sum_{k}a_{k}x^{k}.\eqno(2.3)$$
For $i<0$ or $i>\phi(pqr)=(p-1)(q-1)(r-1)$, $j<0$ or
$j>(p-1)(q-1)(r-1)+1$ and $k<0$ or $k>(p-1)(q-1)(r-1)+1+pq$, we set
$c_{i}=b_{j}=a_{k}=0$. Obviously, we have
$$c_{i}=\sum_{j\geqslant i+1}b_{j}=\sum_{j\geqslant i+1}\;\sum_{\substack{k\equiv j\!\!\!\!\pmod{pq}\\
k\geqslant j+pq}}a_{k}=\sum_{k\geqslant
i+1+pq}a_{k}+\sum_{k\geqslant i+1+2pq}a_{k}+\cdots.\eqno(2.4)$$
Let$$\Phi_{pq}(x)=\sum_{m}d_{m}x^{m},\eqno(2.5)$$ then
$$g(x)=\Phi_{pq}(x^{r})(x^{p}-1)(x^{q}-1)=\sum_{m}d_{m}x^{mr}(x^{p+q}-x^{q}-x^{p}+1).\eqno(2.6)$$
For $m<0$ or $m>\phi(pq)=(p-1)(q-1)$, we set $d_{m}=0$.

\textbf{Notation}\quad $\forall n\in \mathbb{Z}$, let $\overline{n}$
be the unique integer such that $0\leqslant \overline{n}\leqslant
pq-1$ and $\overline{n}\equiv n\pmod{pq}$.

For any $n\in \mathbb{Z}$, define a
map$$\chi_{n}:\mathbb{Z}\longrightarrow \{0,\pm 1\}$$ by
$$\chi_{n}(i)=\left \{\begin{array}{lll} 1 & & \textrm{if there exists an integer } s_{1} \textrm{ with } n+p+q\geqslant i+1+s_{1}pq> n+q\\
-1 & & \textrm{if there exists an integer } s_{2} \textrm{ with }
n+p\geqslant i+1+s_{2}pq> n\\
0 & & \textrm{otherwise}.\end{array}\right.$$ Note that this map is
well-defined. An elementary somewhat tedious argument then shows
that alternatively one can define $\chi_{n}$ by
$$\chi_{n}(i)=\left \{\begin{array}{lll} 1 & & \textrm{if }\overline{n+p+q}\geqslant
\overline{i+1}>\overline{n+q}\textrm{ or }\overline{i+1}\leqslant \overline{n+p+q}<\overline{n+q}\\
   & & \text{or }\overline{n+p+q}<\overline{n+q}<\overline{i+1},\\
-1 & &  \textrm{if }\overline{n+p}\geqslant
\overline{i+1}>\overline{n}\textrm{ or }\overline{i+1}\leqslant \overline{n+p}<\overline{n}\\
   & &  \textrm{or }\overline{n+p}<\overline{n}<\overline{i+1},\\
0  & &  otherwise.
\end{array}\right.$$
Now it is not difficult to verify the lemma below.
\begin{lemma}
With notation as above, we have $$c_{i}=\sum_{mr+p+q\geqslant
i+1+pq}d_{m}\chi_{mr}(i).\eqno(2.7)$$
\end{lemma}
\begin{proof}
Combining (2.3) and (2.6) yields
$$g(x)=\sum_{k}a_{k}x^{k}=\sum_{m}d_{m}x^{mr}(x^{p+q}-x^{q}-x^{p}+1).$$
By (2.4), we know that to compute $c_{i}$ it suffices to consider
only the coefficients $a_{k}$ of the terms of $g(x)$ with exponents
$k\geqslant i+1+pq$. On the other hand, for
$d_{m}x^{mr}(x^{p+q}-x^{q}-x^{p}+1)$, $mr+p+q\geqslant i+1+pq$, the
contribution to $c_{i}$ is
$$d_{m}(W_{i}(mr+p+q)-W_{i}(mr+q)-W_{i}(mr+p)+W_{i}(mr)),\eqno(2.8)$$
where $W_{i}(m_{1})$ counts the number of integers $s\geqslant 1$
such that $m_{1}\geqslant i+1+spq$. Now note that
$$W_{i}(mr+p+q)-W_{i}(mr+q)=\left\{
\begin{array}{lll}1 & & \textrm{if there exists an integer } s_{1} \textrm{ with}\\
 & & mr+p+q\geqslant i+1+s_{1}pq>mr+q;\\
0 & & otherwise,\end{array}\right.$$ and
$$-W_{i}(mr+p)+W_{i}(mr)=\left\{ \begin{array}{lll}-1 & & \textrm{if there exists an integer } s_{2} \textrm{ with}\\
  & & mr+p\geqslant i+1+s_{2}pq>mr;\\
0 & & otherwise.\end{array}\right.$$ By the definition of
$\chi_{n}$, it then follows that the expression in (2.8) equals
$d_{m}\chi_{mr}(i)$, so we complete the proof of the lemma.
\end{proof}

Especially, note that $c_{i}=0$ for $i<0$, so we can immediately
obtain the following consequence which will be very important in the
next section.
\begin{lemma}
For any integer $i$, $$\sum_{m}d_{m}\chi_{mr}(i)=0.\eqno(2.9)$$
\end{lemma}
\begin{proof}
From either definition of $\chi_{n}$, it is easy to find that the
value of $\chi_{n}(i)$ only depends on $\overline{n}$ and
$\overline{i}$. That means that for any $n',i'\in
\mathbb{Z},n'\equiv n\pmod{pq},i'\equiv i\pmod{pq}$, we have
$$\chi_{n'}(i')=\chi_{n}(i).\eqno(2.10)$$  For any integer $i$, there
exists an integer $s$ such that $mr+p+q\geqslant (i-spq)+1+pq$ for
any non-negative integer $m$. Observe that $c_{i}=0$ for $i<0$ and
$d_{m}=0$ for $m<0$, hence we have
\begin{eqnarray*}
\sum_{m}d_{m}\chi_{mr}(i) & = &
\sum_{m}d_{m}\chi_{mr}(i-spq)\\
& = & \sum_{mr+p+q\geqslant
(i-spq)+1+pq}d_{m}\chi_{mr}(i-spq)\\
& = & c_{i-spq}\\
& = & 0.
\end{eqnarray*}
\end{proof}
\begin{lemma}
With the notation as above, we have
$$A(pqr)=\max_{i,j\in\mathbb{Z}}\left|\sum_{m\geqslant
j}d_{m}\chi_{mr}(i)\right|.\eqno(2.11)$$
\end{lemma}
\begin{proof}
By (1.2) and (2.7), obviously we have $$A(pqr)\leqslant
\max_{i,j\in\mathbb{Z}}\left|\sum_{m\geqslant
j}d_{m}\chi_{mr}(i)\right|.\eqno(2.12)$$ Now it suffices to show
that for any $i,j\in\mathbb{Z}$,
$$\left|\sum_{m\geqslant j}d_{m}\chi_{mr}(i)\right|\leqslant
A(pqr).$$ Let $s$ be the largest integer such that $jr+p+q\geqslant
(i+spq)+1+pq$. If $(j-1)r+p+q<(i+spq)+1+pq$,
then$$\left|\sum_{m\geqslant
j}d_{m}\chi_{mr}(i)\right|=\left|\sum_{mr+p+q\geqslant
(i+spq)+1+pq}d_{m}\chi_{mr}(i+spq)\right|=|c_{i+spq}|\leqslant
A(pqr).$$ If $(j-1)r+p+q\geqslant (i+spq)+1+pq$ and
$\chi_{jr}(i)=0$, then $jr+p>(j-1)r+p+q\geqslant (i+spq)+1+pq$
(because $r>q$) and hence $(i+spq)+1+pq\leqslant jr<jr+q<(j+1)r$.
Let $j_{1}$ be the smallest integer such that $j_{1}r+p+q\geqslant
(i+(s+1)pq)+1+pq$, then for $j\leqslant m< j_{1}$, $\chi_{mr}(i)=0$.
Therefore we have
$$\left|\sum_{m\geqslant
j}d_{m}\chi_{mr}(i)\right|=\left|\sum_{m\geqslant
j_{1}}d_{m}\chi_{mr}(i)\right|=\left|c_{i+(s+1)pq}\right|\leqslant
A(pqr).$$ If $(j-1)r+p+q\geqslant (i+spq)+1+pq$ and
$\chi_{jr}(i)\neq 0$, then $\chi_{jr}(i)=-1$, that is,
$jr+p\geqslant (i+spq)+1+pq>jr$. Since $p<q<r$, we get
$(j-2)r+p+q<jr<(i+spq)+1+pq$ and $(j+1)r>jr+p\geqslant (i+spq)+1+pq$
which implies that for $j+1\leqslant m<j_{1}$, $\chi_{mr}(i)=0$. It
follows that
$$\left|\sum_{m\geqslant
j+1}d_{m}\chi_{mr}(i)\right|=\left|\sum_{m\geqslant
j_{1}}d_{m}\chi_{mr}(i)\right|=\left|c_{i+(s+1)pq}\right|\leqslant
A(pqr),$$ and $$\left|\sum_{m\geqslant
j-1}d_{m}\chi_{mr}(i)\right|=\left|c_{i+spq}\right|\leqslant
A(pqr).$$ If $\chi_{(j-1)r}(i)=0$, then $$\left|\sum_{m\geqslant
j}d_{m}\chi_{mr}(i)\right|=\left|\sum_{m\geqslant
j-1}d_{m}\chi_{mr}(i)\right|=\left|c_{i+spq}\right|\leqslant
A(pqr).$$ If $\chi_{(j-1)r}(i)\neq 0$, then $\chi_{(j-1)r}(i)=1$. By
Proposition 1.2, we have $d_{j-1}d_{j}\leqslant 0$, hence
$d_{j-1}\chi_{(j-1)r}(i)d_{j}\chi_{jr}(i)\geqslant 0$, therefore
$$\left|\sum_{m\geqslant j}d_{m}\chi_{mr}(i)\right|\leqslant
\max\left\{\left|\sum_{m\geqslant j-1}d_{m}\chi_{mr}(i)\right|,
\left|\sum_{m\geqslant j+1}d_{m}\chi_{mr}(i)\right|\right\}\leqslant
A(pqr).$$ This completes the proof of the corollary.
\end{proof}

\begin{remark}
If $q$ and $r$ interchange, we will have similar arguments as above.
Lemma 2.1 and Lemma 2.2 still hold, but Lemma 2.3 should be
modified. We can only get the trivial conclusion (2.12), but it is
sufficient for estimating the upper bound of $A(pqr)$ to consider
$\max_{i,j\in\mathbb{Z}}\left|\sum_{m\geqslant
j}d_{m}\chi_{mr}(i)\right|.$
\end{remark}

Based on the results above, we can establish explicitly the
following Theorem 2.5 and Theorem 2.6 which have been proven by
Kaplan [9].
\begin{theorem}[Nathan Kaplan, 2007]
Let $p<q<r$ be  odd primes. Then for any prime $s>q$ such that
$s\equiv \pm r\pmod{pq}$, $A(pqr)=A(pqs)$.
\end{theorem}
\begin{proof}
If $s\equiv r\pmod{pq}$, then $\chi_{mr}(i)=\chi_{ms}(i)$, by
(2.11), we obtain
$$A(pqr)=\max_{i,j\in\mathbb{Z}}\left|\sum_{m\geqslant
j}d_{m}\chi_{mr}(i)\right|=\max_{i,j\in\mathbb{Z}}\left|\sum_{m\geqslant
j}d_{m}\chi_{ms}(i)\right|=A(pqs).$$ Next we consider the case
$s\equiv -r\pmod{pq}$. From the definition of $\chi_{n}$, we can
simply verify that
$$\chi_{mr}(i)=-\chi_{-mr}(-i+p+q-1).\eqno(2.13)$$
Therefore, by (2.10) and (2.13), we have
\begin{eqnarray*}
A(pqr) & = & \max_{i,j\in\mathbb{Z}}\left|\sum_{m\geqslant
j}d_{m}\chi_{mr}(i)\right|\\
  & = & \max_{i,j\in\mathbb{Z}}\left|\sum_{m\geqslant
j}d_{m}\chi_{-mr}(-i+p+q-1)\right|\\
  & = & \max_{i,j\in\mathbb{Z}}\left|\sum_{m\geqslant
j}d_{m}\chi_{ms}(-i+p+q-1)\right|\\
  & = & A(pqs).
\end{eqnarray*}
\end{proof}

\begin{theorem}[Nathan Kaplan, 2007]
Let $p<q$ and $r\equiv \pm 1\pmod{pq}$ be odd primes. Then
$A(pqr)=1.$
\end{theorem}
\begin{proof}
Since $d_{m}=0$ for $m<0$ or $m>\phi(pq)$, by (2.9), we can get for
any pair of integers $i$ and $j\leqslant 0$ or $j>\phi(pq)$,
$$\sum_{m\geqslant j}d_{m}\chi_{mr}(i)=0.$$ Given $i$ and $0<j\leqslant
\phi(pq)$, let $$M^{+}=\{0\leqslant m\leqslant
\phi(pq)|\chi_{mr}(i)=1\},\quad M^{-}=\{0\leqslant m\leqslant
\phi(pq)|\chi_{mr}(i)=-1\}.$$ Since $r\equiv \pm 1\pmod{pq}$, the
definition of $\chi_{n}$ implies that both of $M^{+}$ and $M^{-}$
are sets of consecutive integers, of cardinality at most $p$. Let us
first assume that $j\in M^{-}$, then $j\notin M^{+}$. It follows
that $\chi_{mr}(i)\neq 1$ either for all $\phi(pq)\geqslant
m\geqslant j$ or for all $0\leqslant m<j$, hence $\sum_{m\geqslant
j, m\in M^{+}}d_{m}\chi_{mr}(i)$ or $\sum_{m<j, m\in
M^{+}}d_{m}\chi_{mr}(i)$ should be $0$. On the other hand, by
Proposition 1.2, we have
$$\left|\sum_{m\geqslant j, m\in
M^{+}}d_{m}\chi_{mr}(i)\right|\leqslant 1,\qquad\left|\sum_{m<j,
m\in M^{+}}d_{m}\chi_{mr}(i)\right|\leqslant 1,\eqno(2.14)$$
$$\left|\sum_{m\geqslant j, m\in
M^{-}}d_{m}\chi_{mr}(i)\right|\leqslant 1,\qquad\left|\sum_{m<j,
m\in M^{-}}d_{m}\chi_{mr}(i)\right|\leqslant 1.\eqno(2.15)$$
Combining the above observations and (2.9), we have
\begin{eqnarray*}\left|\sum_{m\geqslant
j}d_{m}\chi_{mr}(i)\right| & = & \left|\sum_{m\geqslant j, m\in
M^{+}}d_{m}\chi_{mr}(i)+\sum_{m\geqslant j, m\in
M^{-}}d_{m}\chi_{mr}(i)\right|\\
 & = & \left|\sum_{m<j, m\in
M^{+}}d_{m}\chi_{mr}(i)+\sum_{m<j, m\in
M^{-}}d_{m}\chi_{mr}(i)\right|\\
 & \leqslant & 1.\end{eqnarray*} For the case $j\notin
M^{-}$, the argument is similar. Therefore by (2.12), we have
$A(pqr)\leqslant 1$, thus $A(pqr)=1$. This completes the proof.
\end{proof}

\begin{theorem}
Let $p<q<r$ be odd primes. Then $A(pqr)\leqslant \min\{\overline{r},
pq-\overline{r}\}.$
\end{theorem}
\begin{proof}
Given $i$ and $0<j\leqslant \phi(pq)$, according to the proof of
Theorem 2.6, there must exist a partition of $[0, \phi(pq)]$,
$0=t_{0}<t_{1}\leqslant t_{2}\leqslant \cdots\leqslant
t_{\overline{r}-1}\leqslant t_{\overline{r}}=\phi(pq), t_{k}\in
\mathbb{Z}$ for $0\leqslant k\leqslant \overline{r}$, such that
$$M_{k}^{+}=\{t_{k-1}<m\leqslant t_{k}|\chi_{mr}(i)=1\}, 2\leqslant
k\leqslant \overline{r},$$
$$M_{1}^{+}=\{t_{0}\leqslant m\leqslant
t_{1}|\chi_{mr}(i)=1\}$$ are all sets of consecutive integers, of
cardinality at most $p$. In fact, we can obtain this partition by
induction. First, let $m_{1}$ be the smallest integer such that
$0\leqslant m_{1}\leqslant \phi(pq)$ and $\chi_{m_{1}r}(i)=1$. Then
we can take $t_{1}+1$ to equal the smallest integer such that
$m_{1}\leqslant t_{1}\leqslant \phi(pq)$ and
$\chi_{(t_{1}+1)r}(i)\neq 1$. Next let $m_{2}$ be the smallest
integer such that $t_{1}<m_{2}\leqslant \phi(pq)$ and
$\chi_{m_{2}r}(i)=1$. Then we can take $t_{2}+1$ to equal the
smallest integer such that $m_{2}\leqslant t_{2}\leqslant \phi(pq)$
and $\chi_{(t_{2}+1)r}(i)\neq 1$. Moreover, by the definition of
$\chi_{n}$, we have
$$(m_{2}-m_{1})\overline{r}\geqslant p(q-1).\eqno(2.16)$$
Inductively, we can get $m_{3}, t_{3}, \cdots, m_{\overline{r}},
t_{\overline{r}}$. Notice that if $m_{k}$ does not exist or
$m_{k}=\phi(pq)$, then we can take
$t_{k}=t_{k+1}=\cdots=t_{\overline{r}}=\phi(pq)$. Specially, if
$t_{\overline{r}}<\phi(pq)$, we claim $m_{\overline{r}+1}$ does not
exist. Otherwise, by (2.16) we have
$$(m_{\overline{r}+1}-m_{1})\overline{r}=(m_{\overline{r}+1}-m_{\overline{r}}+\cdots+m_{2}-m_{1})\overline{r}\geqslant p(q-1)\overline{r}.\eqno(2.17)$$
On the other hand,
$$(m_{\overline{r}+1}-m_{1})\overline{r}\leqslant \phi(pq)\overline{r}=(p-1)(q-1)\overline{r}.$$
This contradicts (2.17), so we can always take
$t_{\overline{r}}=\phi(pq)$. Similarly, there also exists a
partition of $[0, \phi(pq)]$, $0=s_{0}<s_{1}\leqslant s_{2}\leqslant
\cdots\leqslant s_{\overline{r}-1}\leqslant
s_{\overline{r}}=\phi(pq), s_{l}\in \mathbb{Z}$ for $0\leqslant
l\leqslant \overline{r}$, such that
$$M_{l}^{-}=\{s_{l-1}<m\leqslant s_{l}|\chi_{mr}(i)=-1\}, 2\leqslant
l\leqslant \overline{r},$$
$$M_{1}^{-}=\{s_{0}\leqslant m\leqslant
s_{1}|\chi_{mr}(i)=-1\}$$ are all sets of consecutive integers, of
cardinality at most $p$.

Assume $t_{k-1}<j\leqslant t_{k}$ for some $1\leqslant k\leqslant
\overline{r}$ and $s_{l-1}<j\leqslant s_{l}$ for some $1\leqslant
l\leqslant \overline{r}$. Let us first assume $j\in M_{l}^{-}$, then
$j\notin M_{k}^{+}$. By (2.14) and (2.15), we have
\begin{eqnarray*}
\sum_{m\geqslant j}d_{m}\chi_{mr}(i) & = & \sum_{m\geqslant j, m\in
M_{k}^{+}}d_{m}\chi_{mr}(i)+\cdots
+\sum_{m\geqslant j, m\in M_{\overline{r}}^{+}}d_{m}\chi_{mr}(i)\\
 & & +\sum_{m\geqslant j, m\in M_{l}^{-}}d_{m}\chi_{mr}(i)+\cdots
+\sum_{m\geqslant j, m\in M_{\overline{r}}^{-}}d_{m}\chi_{mr}(i)\\
\left|\sum_{m\geqslant j}d_{m}\chi_{mr}(i)\right| & \leqslant &
\left|\sum_{m\geqslant j, m\in
M_{k}^{+}}d_{m}\chi_{mr}(i)\right|+\cdots+\left|\sum_{m\geqslant
j, m\in M_{\overline{r}}^{+}}d_{m}\chi_{mr}(i)\right|\\
 & & +\left|\sum_{m\geqslant j, m\in
 M_{l}^{-}}d_{m}\chi_{mr}(i)\right|+\cdots+\left|\sum_{m\geqslant j, m\in
 M_{\overline{r}}^{-}}d_{m}\chi_{mr}(i)\right|\\
 & \leqslant & 2\overline{r}-k-l+2.
\end{eqnarray*}
Similarly we also have
\begin{eqnarray*}
\left|\sum_{m<j}d_{m}\chi_{mr}(i)\right| & \leqslant &
\left|\sum_{m<j, m\in
M_{1}^{+}}d_{m}\chi_{mr}(i)\right|+\cdots+\left|\sum_{m<j, m\in M_{k-1}^{+}}d_{m}\chi_{mr}(i)\right|\\
 & & +\left|\sum_{m<j, m\in
 M_{1}^{-}}d_{m}\chi_{mr}(i)\right|+\cdots+\left|\sum_{m<j, m\in
 M_{l}^{-}}d_{m}\chi_{mr}(i)\right|\\
 & \leqslant & k+l-1.
\end{eqnarray*}
Thus $$\left|\sum_{m\geqslant
j}d_{m}\chi_{mr}(i)\right|+\left|\sum_{m<j}d_{m}\chi_{mr}(i)\right|\leqslant
2\overline{r}+1.$$ By (2.9), we certainly get
$$\left|\sum_{m\geqslant j}d_{m}\chi_{mr}(i)\right|=\left|\sum_{m<j}d_{m}\chi_{mr}(i)\right|.$$
Therefore $$\left|\sum_{m\geqslant
j}d_{m}\chi_{mr}(i)\right|\leqslant \overline{r}.$$ For the case
$j\notin M_{l}^{-}$, the argument is similar. Therefore (2.11)
yields $A(pqr)\leqslant \overline{r}$. On the other hand, by
Dirichlet's Prime Number Theorem, we know there exists a prime $s>q$
satisfying $\overline{s}=pq-\overline{r}$. That means $s\equiv
-r\pmod{pq}$, by Theorem 2.5 and the arguments above, we get
$$A(pqr)=A(pqs)\leqslant \overline{s}=pq-\overline{r}.$$
We have thus proved the theorem.
\end{proof}

\section{Main result}
Now to estimate the upper bound of $A(pqr)$, we need to investigate
the properties of the coefficients of $\Phi_{pq}$. First we
introduce some notation for the rest of the paper.

\textbf{Notation}\quad For any distinct primes $p$ and $q$, let
$q_{p}^{*}$ be the unique integer such that $0<q_{p}^{*}<p$ and
$qq_{p}^{*}\equiv 1\pmod{p}$. Let $\overline{q_{p}}$ be the unique
integer such that $0<\overline{q_{p}}<p$ and $q\equiv
\overline{q_{p}}\pmod{p}$.

About the coefficients of $\Phi_{pq}$, Lam and Leung [10] showed
\begin{theorem}[T.Y. Lam and K.H. Leung, 1996]
Let $\Phi_{pq}(x)=\sum_{m}d_{m}x^{m}$. For $0\leqslant m\leqslant
\phi(pq)$, we have

(A) $d_{m}=1$ if and only if $m=up+vq$ for some $u\in [0,
p_{q}^{*}-1]$ and $v\in [0, q_{p}^{*}-1]$;

(B) $d_{m}=-1$ if and only if $m+pq=u'p+v'q$ for some $u'\in
[p_{q}^{*}, q-1]$ and $v'\in [q_{p}^{*}, p-1]$;

(C) $d_{m}=0$ otherwise.

The numbers of terms of the former two kinds are, respectively,
$p_{q}^{*}q_{p}^{*}$ and $(q-p_{q}^{*})(p-q_{p}^{*})$, with
difference $1$ since $(p-1)(q-1)=(p_{q}^{*}-1)p+(q_{p}^{*}-1)q$.
\end{theorem}

About $A(pqr)$, the best known general upper bound to date is due to
Bart{\l}omiej Bzd\c{e}ga [5]. He gave the following important result
\begin{theorem}[Bart{\l}omiej Bzd\c{e}ga, 2008]
Set $$\alpha=\min\{q_{p}^{*}, r_{p}^{*}, p-q_{p}^{*},
p-r_{p}^{*}\}$$ and $0<\beta<p$ satisfying $\alpha\beta qr\equiv
1\pmod{p}$. Put $\beta^{*}=\min\{\beta, p-\beta\}$. Then we have
$$A(pqr)\leqslant\min\{2\alpha+\beta^{*}, p-\beta^{*}\}.$$
\end{theorem}

We can now prove our main result.
\begin{theorem}
Let $p<q<r$ be odd primes. Suppose $\min\{\overline{q_{p}},
p-\overline{q_{p}}, \overline{r_{p}}, p-\overline{r_{p}}\}>
\frac{p-1}{3}$, then $A(pqr)\leqslant \frac{p+\beta^{*}}{2}$.
\end{theorem}
\begin{proof}
Let us first assume $$1\leqslant p-q_{p}^{*}\leqslant
r_{p}^{*}<p-r_{p}^{*}\leqslant q_{p}^{*}\leqslant p-1.\eqno(3.1)$$
According to Theorem 3.2, it follows $\alpha=p-q_{p}^{*}$,
$\beta=p-r_{p}^{*}$ and $\beta^{*}=r_{p}^{*}$. Suppose
$A(pqr)>\frac{p+\beta^{*}}{2}$, so we easily get
$$\frac{p+\beta^{*}}{2}<p-\beta^{*}.$$ This implies that $$\beta^{*}<\frac{1}{3}p.\eqno(3.2)$$
By (2.12), we know there exist a pair of integers $i, j$ such that
$$\left|\sum_{m\geqslant
j}d_{m}\chi_{mr}(i)\right|>\frac{p+\beta^{*}}{2}.\eqno(3.3)$$ By
Theorem 3.1, we can divide the nonzero terms of $\Phi_{pq}(x)$ into
$p$ classes depending on the value of $v$ or $v'$. From the
definition of $\chi_{n}$, we can simply verify that for any given
class, there is at most one term such that $\chi_{mr}(i)=1$. For the
case $\chi_{mr}(i)=-1$, we have the similar result. By (2.9) and
(3.3), we immediately obtain
$$\left|\sum_{m<j}d_{m}\chi_{mr}(i)\right|>\frac{p+\beta^{*}}{2}.\eqno(3.4)$$ This
implies that the number of the nonzero terms of $\Phi_{pq}(x)$ such
that $\chi_{mr}(i)=\pm 1$ is more than $p+\beta^{*}$. Therefore
there are more than $\beta^{*}$ classes such that each of them has
two terms $d_{m}x^{m}$ and $d_{m'}x^{m'}$ such that $\chi_{mr}(i)=1$
and $\chi_{m'r}(i)=-1$ respectively. Moreover, $m$ and $m'$ should
satisfy $m\geqslant j, m'<j$ or $m<j, m'\geqslant j$, otherwise
$d_{m}\chi_{mr}(i)+d_{m'}\chi_{m'r}(i)=0$, thus their contributions
to the left sides of (3.3) and (3.4) are both zero. For convenience
of description, We call them the special classes.

Now we claim that the special classes contain not only the classes
of $d_{m}=1$, but also the classes of $d_{m}=-1$. In fact, by
Theorem 3.1, the number of the classes of $d_{m}=-1$ is just
$p-q_{p}^{*}\leqslant \beta^{*}$, so the special classes must
contain the classes of $d_{m}=1$. If there are more than $\beta^{*}$
classes of $d_{m}=1$ in the special classes, then it implies that
there exist $u_{1}\in[0, p_{q}^{*}-1], v_{1}\in[0, q_{p}^{*}-1]$
such that
$$u_{1}p+v_{1}q\geqslant j,$$ $u_{2}\in[0, p_{q}^{*}-1],
v_{2}\in[0, q_{p}^{*}-1]$ such that
$$u_{2}p+v_{2}q<j$$ and $$v_{2}-v_{1}\geqslant \beta^{*}.$$
This yields
$$(u_{1}-u_{2})p+(v_{1}-v_{2})q>0,$$hence
$$(u_{1}-u_{2})p>(v_{2}-v_{1})q\geqslant \beta^{*}q.$$
On the other hand,
$$(u_{1}-u_{2})p\leqslant
(p_{q}^{*}-1)p=(p-q_{p}^{*})q-p+1\leqslant \beta^{*}q-p+1.$$ The
equality holds because $(p-1)(q-1)=(p_{q}^{*}-1)p+(q_{p}^{*}-1)q$.
Therefore we derive a contradiction and prove our claim.

Since $(up+vq)r+p\equiv (up+(v-r_{p}^{*})q)r+p+q\pmod{pq}$, we have
$$\chi_{mr}(i)=-1\Longleftrightarrow \chi_{(m-r_{p}^{*}q)r}(i)=1.\eqno(3.5)$$
We claim that $$\sum_{m\geqslant
j}d_{m}\chi_{mr}(i)<-\frac{p+\beta^{*}}{2}.\eqno(3.6)$$ By (3.3), we
know $\sum_{m\geqslant j}d_{m}\chi_{mr}(i)>\frac{p+\beta^{*}}{2}$ or
$\sum_{m\geqslant j}d_{m}\chi_{mr}(i)<-\frac{p+\beta^{*}}{2}$. If
the former holds, then $$\sum_{m\geqslant j,
d_{m}=1}d_{m}\chi_{mr}(i)>\frac{p-\beta^{*}}{2}>\beta^{*}\eqno(3.7)$$
because
$$\sum_{m\geqslant j, d_{m}=-1}d_{m}\chi_{mr}(i)\leqslant
\beta^{*}.$$ Thus there must exist $u\in[0, p_{q}^{*}-1]$ and
$v\in[0, q_{p}^{*}-1-\beta^{*}]$ such that $\chi_{(up+vq)r}(i)=1$.
By (3.5), we have $\chi_{(up+(v+r_{p}^{*})q)r}(i)=-1$ and
$v+r_{p}^{*}\in[0, q_{p}^{*}-1]$ since $\beta^{*}=r_{p}^{*}$. Hence
$d_{up+vq}\chi_{(up+vq)r}(i)+d_{up+(v+r_{p}^{*})q}\chi_{(up+(v+r_{p}^{*})q)r}(i)=0$,
their contributions to the left side of (3.7) are zero. This is a
contradiction, so we establish the second claim.

Combining the above arguments yields there exist $u_{1}, u_{2}\in[0,
p_{q}^{*}-1]$ and $v\in[0, q_{p}^{*}-1]$ such that
$u_{1}p+vq\geqslant j>u_{2}p+vq$, $\chi_{(u_{1}p+vq)r}(i)=-1$ and
$\chi_{(u_{2}p+vq)r}(i)=1$. This implies that
$$\overline{(u_{1}p+vq)r+p+\overline{q_{p}}}=\overline{(u_{2}p+vq)r+p+q}\eqno(3.8)$$
or
$$\overline{(u_{1}p+vq)r+p-(p-\overline{q_{p}})}=\overline{(u_{2}p+vq)r+p+q}.\eqno(3.9)$$
Similarly, we also have there exist $u'_{1}, u'_{2}\in[p_{q}^{*},
q-1]$ and $v'\in[q_{p}^{*}, p-1]$ such that $u'_{1}p+v'q-pq\geqslant
j>u'_{2}p+v'q-pq$, $\chi_{(u'_{1}p+v'q-pq)r}(i)=1$ and
$\chi_{(u'_{2}p+v'q-pq)r}(i)=-1$. This implies that
$$\overline{(u'_{1}p+v'q-pq)r+p+q-\overline{q_{p}}}=\overline{(u'_{2}p+v'q-pq)r+p}\eqno(3.10)$$
or
$$\overline{(u'_{1}p+v'q-pq)r+p+q+(p-\overline{q_{p}})}=\overline{(u'_{2}p+v'q-pq)r+p}.\eqno(3.11)$$
If (3.8) and (3.10) hold simultaneously, then we get
$$\overline{(u_{1}+u'_{1})pr}=\overline{(u_{2}+u'_{2})pr}.$$
Hence $$q\mid (u_{1}+u'_{1}-u_{2}-u'_{2}).$$ This contradicts
$0<u_{1}+u'_{1}-u_{2}-u'_{2}\leqslant q-2$. Similarly (3.9) and
(3.11) can not hold simultaneously, so without loss of generality,
we assume (3.9) and (3.10) are correct. By (3.5), we have
$\chi_{(u'_{2}p+v'q-pq-r_{p}^{*}q)r}(i)=1$ since
$\chi_{(u'_{2}p+v'q-pq)r}(i)=-1$. It follows that the class of
$v'-r_{p}^{*}\in[0, q_{p}^{*}-1]$ does not contain a term such that
$\chi_{mr}(i)=1$ since $\chi_{((u'_{2}-q)p+(v'-r_{p}^{*})q)r}(i)=1$.
If it does not contain a term such that $\chi_{mr}(i)=-1$ either,
then the contributions of this class to the left sides of (3.3) and
(3.4) are both zero. It is easy to verify that there must exist a
special class of $v'\in[q_{p}^{*}, p-1]$ such that the class of
$v'-r_{p}^{*}$ contains a term such that $m\geqslant j$ and
$\chi_{mr}(i)=-1$. This implies that there exist $u_{3}\in[0,
p_{q}^{*}-1]$ such that $u_{3}p+(v'-r_{p}^{*})q\geqslant j$ and
$\chi_{(u_{3}p+(v'-r_{p}^{*})q)r}(i)=-1$, so
$$\overline{(u_{3}p+(v'-r_{p}^{*})q)r+p+\overline{q_{p}}}=\overline{((u'_{2}-q)p+(v'-r_{p}^{*})q)r+p+q}\eqno(3.12)$$
or
$$\overline{(u_{3}p+(v'-r_{p}^{*})q)r+p-(p-\overline{q_{p}})}=\overline{((u'_{2}-q)p+(v'-r_{p}^{*})q)r+p+q}.$$
If the latter holds, by (3.9) we get
$$\overline{(u_{3}-u_{1})pr}=\overline{(u'_{2}-q-u_{2})pr}.$$
Hence $$q\mid (u_{3}+u_{2}-u_{1}-u'_{2}).\eqno(3.13)$$ On the other
hand, by $$u_{3}p+(v'-r_{p}^{*})q\geqslant j>u'_{2}p+v'q-pq$$ we
have $$0>(u_{3}-u'_{2})p>(r_{p}^{*}-p)q.$$ Note that
$$0>(u_{2}-u_{1})p\geqslant
-(p_{q}^{*}-1)p=-(p-q_{p}^{*})q+p-1\geqslant -r_{p}^{*}q+p-1,$$ so
we can get $$0>(u_{3}+u_{2}-u_{1}-u'_{2})p>-pq+p-1.$$ This
contradicts (3.13) and establishes the validity of (3.12).

Similarly we have $\chi_{(u'_{1}p+v'q-pq+r_{p}^{*}q)r}(i)=-1$
because $\chi_{(u'_{1}p+v'q-pq)r}(i)=1$. The class of
$v'-p+r_{p}^{*}\in[0, q_{p}^{*}-1]$ does not contain a term such
that $\chi_{mr}(i)=-1$, but it contains a term such that
$\chi_{mr}(i)=1$. This implies that there exist $u_{4}\in[0,
p_{q}^{*}-1]$ such that $u_{4}p+(v'-p+r_{p}^{*})q<j$ and
$\chi_{(u_{4}p+(v'-p+r_{p}^{*})q)r}(i)=1$, so we can get
$$\overline{(u'_{1}p+v'q-pq+r_{p}^{*}q)r+p+\overline{q_{p}}}=\overline{(u_{4}p+(v'-p+r_{p}^{*})q)r+p+q}.\eqno(3.14)$$
Combining (3.10), (3.12) and (3.14), we obtain
$$\overline{(u_{3}p+(v'-r_{p}^{*})q)r+p+3\overline{q_{p}}}=\overline{(u_{4}p+(v'-p+r_{p}^{*})q)r+p+q}.$$
Since $\chi_{(u_{3}p+(v'-r_{p}^{*})q)r}(i)=-1$ and
$\chi_{(u_{4}p+(v'-p+r_{p}^{*})q)r}(i)=1$, we know
$$\overline{((u_{4}p+(v'-p+r_{p}^{*})q)r+p+q)-((u_{3}p+(v'-r_{p}^{*})q)r+p)}\leqslant p-1$$
or
$$\overline{((u_{3}p+(v'-r_{p}^{*})q)r+p)-((u_{4}p+(v'-p+r_{p}^{*})q)r+p+q)}\leqslant p-1.$$
That is, $$3\overline{q_{p}}\leqslant p-1$$ or
$$pq-3\overline{q_{p}}\leqslant p-1.$$ Note that $0<\overline{q_{p}}<p$, hence it is obvious that the former
inequality holds. Therefore $\overline{q_{p}}\leqslant
\frac{p-1}{3}$.

If (3.8) and (3.11) hold simultaneously, we similarly get
$p-\overline{q_{p}}\leqslant \frac{p-1}{3}$.

For the cases $$1\leqslant q_{p}^{*}\leqslant
r_{p}^{*}<p-r_{p}^{*}\leqslant p-q_{p}^{*}\leqslant p-1,$$
$$1\leqslant p-q_{p}^{*}\leqslant p-r_{p}^{*}<r_{p}^{*}\leqslant
q_{p^{*}}\leqslant p-1$$ and $$1\leqslant q_{p}^{*}\leqslant
p-r_{p}^{*}<r_{p}^{*}\leqslant p-q_{p}^{*}\leqslant p-1,$$ we can
get the above results similarly. Observe that, we can immediately
obtain the remaining four cases provided that $q_{p}^{*}$ and
$r_{p}^{*}$ interchange. In these cases, by Remark 2.4, it is not
difficult to establish $\overline{r_{p}}\leqslant\frac{p-1}{3}$ or
$p-\overline{r_{p}}\leqslant\frac{p-1}{3}$ similarly. Combining the
above arguments yields $$\min\{\overline{q_{p}}, p-\overline{q_{p}},
\overline{r_{p}}, p-\overline{r_{p}}\}\leqslant \frac{p-1}{3}.$$
This is a contradiction and completes the proof of the theorem.
\end{proof}

\begin{corollary}
Let $p<q<r$ be odd primes. Suppose $\min\{\overline{q_{p}},
p-\overline{q_{p}}, \overline{r_{p}}, p-\overline{r_{p}}\}>
\frac{p-1}{3}$, then $A(pqr)\leqslant \frac{2}{3}p$.
\end{corollary}
\begin{proof}
By Theorem 3.2 and 3.3, we have
$$A(pqr)\leqslant\min\{p-\beta^{*}, \frac{p+\beta^{*}}{2}\}\leqslant\frac{2}{3}p.$$
\end{proof}

Now we can show in the special case $p=7$ that both Beiter's
conjecture and the Corrected Beiter conjecture are correct.
\begin{theorem}
We have $M(7)=4$.
\end{theorem}
\begin{proof}
Suppose there exists a pair of primes $7<q<r$ such that
$A(7qr)\geqslant 5$. Then by Theorem 3.2, we must have
$\alpha=\beta^{*}=2$. It implies that $q, r\equiv \pm 3\pmod{7}$,
hence $\min\{\overline{q_{7}}, 7-\overline{q_{7}}, \overline{r_{7}},
7-\overline{r_{7}}\}> \frac{7-1}{3}=2$. By Theorem 3.3, we have
$A(7qr)\leqslant \frac{7+2}{2}<5$. This is a contradiction, so
$M(7)\leqslant 4$. Recall that M\"{o}ller [12] indicated that for
any prime $p$ there are infinitely many pairs of primes $q<r$ such
that $A(pqr)\geqslant \frac{p+1}{2}$. Therefore we have $M(7)=4$.
\end{proof}

\textbf{Acknowledgements.} The authors would like to thank the
referee for several helpful suggestions.


\begin{thebibliography}{100000}
\bibitem[1]{G. Bachman} G. Bachman, On the coefficients of ternary
cyclotomic polynomials, J. Number Theory 100 (2003) 104-116.

\bibitem[2]{A.S. Bang} A.S. Bang, Om Lingingen $\Phi_{n}(x)=0$,
Tidsskr. Math. 6 (1895) 6--12.

\bibitem[3]{M. Beiter[3]} M. Beiter, Magnitude of the coefficients of the
cyclotomic polynomial $F_{pqr}$, Amer. Math. Monthly 75 (1968)
370--372.

\bibitem[4]{M. Beiter[4]} M. Beiter, Magnitude of the coefficients of the
cyclotomic polynomial $F_{pqr}$,
\uppercase\expandafter{\romannumeral2}, Duke Math. J. 38 (1971)
591--594.

\bibitem[5]{B. Bzdega} B. Bzd\c{e}ga, Bounds on ternary cyclotomic
coefficients, arXiv:0812.4024, submitted for publication.

\bibitem[6]{L. Carlitz} L. Carlitz, The number of terms in the
cyclotomic polynomial $F_{pq}(x)$, Amer. Math. Monthly 73 (1966)
979-981.

\bibitem[7]{P. Erdos} P. Erd\"{o}s, On the coefficients of the
cyclotomic polynomial, Bull. Amer. Math. Soc. 52 (1946) 179-184.

\bibitem[8]{Y. Gallot and P. Moree} Y. Gallot and P. Moree, Ternary cyclotomic polynomials having a large
coefficient, J. Reine Angew. Math. 632(2009), 105-125.

\bibitem[9]{N. Kaplan}N. Kaplan, Flat cyclotomic polynomials of order
three, J. Number Theory 127 (2007) 118--126.

\bibitem[10]{T.Y. Lam, K.H. Leung} T.Y. Lam, K.H. Leung, On the
cyclotomic polynomial $\Phi_{pq}(X)$, Amer. Math. Monthly 103 (1996)
562-564.

\bibitem[11]{H.W. Lenstra} H.W. Lenstra, Vanishing sums of roots of
unity, in: Proceedings, Bicentennial Congress Wiskundig Genootschap,
Vrije Univ., Amsterdam, 1978,
Part\uppercase\expandafter{\romannumeral2}, 1979, pp. 249-268.

\bibitem[12]{H. Moller} H. M\"{o}ller,
\"{U}ber die Koeffizienten des n-ten Kreisteilungspolynoms, Math. Z.
119 (1971) 33--40.

\bibitem[13]{R.C. Vaughan} R.C. Vaughan, Bounds for the coefficients
of cyclotomic polynomials, Michigan Math J. 21 (1975) 289-295.
\end{thebibliography}
\end{document}